\documentclass[12pt]{article}
\usepackage{a4}
\usepackage{amsmath}
\usepackage{amsthm}
\usepackage{amsfonts}
\usepackage{amssymb}
\usepackage{cite}
\usepackage{hyperref}
\usepackage{epsfig}
\usepackage{colonequals}
\newtheorem{theorem}{Theorem}
\newtheorem{conjecture}{Conjecture}
\newtheorem{question}{Question}
\newtheorem{lemma}[theorem]{Lemma}

\usepackage{tikz}

\newcommand{\E}{\mathbb{E}}
\begin{document}
\title{Independence number in triangle-free graphs avoiding a minor}
\author{%
     Zden\v{e}k Dvo\v{r}\'ak\thanks{Computer Science Institute (CSI) of Charles University,
           Malostransk{\'e} n{\'a}m{\v e}st{\'\i} 25, 118 00 Prague, 
           Czech Republic. E-mail: \protect\href{mailto:rakdver@iuuk.mff.cuni.cz}{\protect\nolinkurl{rakdver@iuuk.mff.cuni.cz}}.
           Supported by project 17-04611S (Ramsey-like aspects of graph coloring) of Czech Science Foundation.}
\and Liana Yepremyan\thanks{Mathematical Institute, University of Oxford, Oxford, UK. 
E-mail: \protect\href{mailto:yepremyan@maths.ox.ac.uk}{\protect\nolinkurl{yepremyan@maths.ox.ac.uk}}. Supported by ERC Consolidator Grant 647678.}}

\date{\today}
\maketitle
\begin{abstract}
The celebrated Hadwiger's conjecture states that if a graph contains no $K_{t+1}$ minor then it is $t$-colourable. 
If true, it would in particular imply that every $n$-vertex $K_{t+1}$-minor-free graph has an independent set of size at least $n/t$.
In 1982, Duchet and Meyniel proved that this bound holds within a factor $2$. Their bound has been improved; most notably in an absolute factor by Fox, which was later improved by Balogh and Kostochka. Here we consider the same question for triangle-free graphs.  By the results of Shearer and Kostochka and Thomason, it follows that any triangle-free graph with no $K_t$ minor has an independent set
of size $\Omega(\tfrac{\sqrt{\log{t}}}{t}n)$. We show that  a much larger independent set exists; for all sufficiently large $t$ every  triangle-free graph on $n$ vertices with no $K_t$-minor has an independent set of size $ \tfrac{n}{t^{1-\varepsilon}}$. This answers a question of Sergey Norin.
\end{abstract}

\section{Introduction}

We say that a graph $H$ is a minor of another graph $G$ if $H$ can be obtained from a subgraph of $G$ by contracting edges. By an $H$-minor of $G$,
we mean a minor of $G$ which is isomorphic to $H$. The famous Hadwiger's conjecture~\cite{hadwiger} from 1943 states that if a graph does not contain $K_{t+1}$ then it is $t$-colorable.  In other words, this conjecture says that  if the chromatic number of a graph is at least~$t$ then
it must contain $K_t$ as a minor. This conjecture is one of the most celebrated ones in graph theory and if true, it would be a far-reaching generalization of the Four Colour Theorem~\cite{FCT}. It was solved by Hadwiger for $t\leq 3$. For $t=4$, Wagner~\cite{wagner} showed it is equivalent to the Four Colour Theorem (long before the latter was proven).
For $t=5$, Robertson, Seymour and Thomas~\cite{roberstonseymthomas}  proved it using the Four Colour Theorem. The conjecture is open for all $t\geq 6$.

Note that Hadwiger's conjecture is known to be true for almost all graphs. Bollob\'{a}s, Catlin, and Erd\H{o}s~\cite{BCE} showed that the random graph $G(n, 1/2)$, where each pair among $n$ vertices is present as an edge with probability $1/2$ does satisfy Hadwiger's conjecture; the largest clique minor of $G(n, 1/2)$ is of order $\Theta(\frac{n}{\sqrt{\log{n}}})$ while it is well known (see, for example~\cite{FriezeKaronski}) that the chromatic number of $G(n,1/2)$ is of size $(1\pm o(1))\frac{n}{2\log_2{n}}$. 

Hadwiger's conjecture is hard; one of the obstacles being that graphs with large chromatic number don't have specific structural characteristics. For this reason, people started looking at various weakenings  of the conjecture. For a graph $G$, $\alpha(G)$ is defined to be the size of the largest independent set. Every $n$-vertex graph $G$ has chromatic number at least $\lceil n/\alpha(G) \rceil$, and should contain a clique minor of this size if Hadwiger's conjecture is true. So for a while the following conjecture, even though not clearly stated as such, has received  quite a lot of attention (see, for example, Seymour's survey on Hadwiger's conjecture~\cite{seymour}.)

\begin{conjecture}\label{conj:stable}
For every positive integer $t$, if an $n$-vertex graph $G$ does not contain $K_{t+1}$ as a minor then $\alpha(G)\geq n/t$.
\end{conjecture}

In 1982 Duchet and Meyniel~\cite{duchetmeyniel} proved that this bound holds within a factor of $2$:
every $n$-vertex graph $G$ which does not contain $K_{t+1}$ as a minor satisfies $\alpha(G)\geq\tfrac{n}{2t}+\frac{1}{2}$. Subsequently this bound has been improved
in various works~\cite{maffraymeyniel, woodall, PlummerStiToft, kawarabayashiplumtoft, kawarabayashisong, Pedersentoft,Baloghlenzwu,fox,baloghkostochka}.
Among these the first multiplicative constant factor improvement was done by Fox~\cite{fox}: every $n$-vertex graph with no $K_{t+1}$ minor satisfies
$\alpha(G)\geq n/(2-c)t$, for some absolute constant $c>0.017$.  Building upon the main idea of set potentials of~\cite{fox},
Balogh and Kostochka~\cite{baloghkostochka} improved the bound on the constant $c$ to $c>1/(19.2)>0.052$.

We are interested in studying Conjecture~\ref{conj:stable} for triangle-free graphs. This question has been initiated by Sergey Norin at Bellairs Research Institute during 2017
Barbados Graph Theory workshop, the motivation ultimately being to prove or disprove Hadwiger's conjecture for triangle-free graphs.  Let us remark that series of works is present in the literature of a similar  flavour. To name a few, K\"uhn and Osthus in~\cite{KuhnOsthusKss} proved that Hadwiger's conjecture is true  for  $K_{s,s}$-free graphs whose chromatic number is sufficiently large compared with $s$ and in~\cite{KuhnOsthusGirth}  for $C_4$-free graphs of sufficiently large chromatic number and for graphs of girth at least~$19$. More recently,  Krivelevich and Sudakov~\cite{KS} studied the existence of complete minors in graphs with good vertex expansion properties and finally, Nenadov and Krivelevich~\cite{NS} studied  the same question for graphs with good edge expansion properties.

Conjecture~\ref{conj:stable} for triangle-free graphs has Ramsey theoretic motivation. For $s,k$ integers, the \emph{Ramsey number} $R(s, k)$ is the smallest integer $n$ such that every graph on $n$ vertices contains a copy of $K_s$ or its complement contains a copy of $K_k$. The determination of Ramsey numbers is very hard and a central topic in Extremal Combinatorics. Ajtai, Koml\'os and Szemer\'edi~\cite{ajtaikomlosszemeredi} proved that any triangle-free graph on $n$ vertices with average degree $d$ has an independent set of size at least $c\tfrac{\log{d}}{d}n$,
for some small constant $c>0$.  The constant $c$ was later improved to $(1+o_d(1))$ by Shearer~\cite{shearer}.  The result of Ajtai, Koml\'os and Szemer\'edi~\cite{ajtaikomlosszemeredi}  and a result of Kim~\cite{kim} show that  $R(3,k) = \Theta(k^2/\log{k})$.  Determining the right constant in the bound for $R(3,k)$ is a major problem in Ramsey theory.  This result of Shearer implies that $ R(3,k) \leq (1+o(1))k^2/\log{k}$. 
Independent works of Bohman and Keevash~\cite{bohmankeevash} and Fiz Pontiveros, Griffiths, and Morris~\cite{FGM}
show that $R(3, k) \geq (1/4 + o(1))k^2/ \log{k}$. Recently,  Davies, Jenssen, Perkins and Roberts~\cite{Jenssenetal} proved a lower bound on the average size of an independent set in a triangle-free graph of maximum degree $d$, matching the asymptotic form of Shearer's result, and in turn giving
an alternative proof of the above upper bound on $R(3, k)$. Their results also implied lower bounds on the total number of independent sets, improving on the previously known best bounds by Cooper and Mubayi~\cite{coopermubayi}.

So what is known about the size of largest independent set in triangle-free graphs without $K_t$-minor? Kostochka~\cite{kostochka} and Thomason~\cite{thomason} independently proved that  the average degree of graphs with no $K_t$ minor is of order $O(t\sqrt{\log{t}})$; with random graphs showing tightness of this bound (later, the exact constant factor was determined by Thomason~\cite{thomasonexact}). So by combining this result together with the result of Ajtai, Koml\'os and Szemer\'edi, we get that in any triangle-free $n$-vertex graph with no $K_t$ minor, one can  always guarantee independent sets of size $\Omega(\tfrac{\sqrt{\log{t}}}{t}n)$.  Our main result shows that much larger independent sets can be guaranteed.

\begin{theorem}
\label{thm:main}
For every positive $0<\varepsilon<1/26$ there exists a positive integer $t_0$ such that for every $t\ge t_0$,
if $G$  is a triangle-free graph on $n$ vertices with no $K_t$-minor then $\alpha(G)\ge \tfrac{n}{t^{1-\varepsilon}}$.
\end{theorem}
We made no attempt to further improve the bound on $\varepsilon$ in this theorem but note that it cannot be larger than $1/3$, due to the following well-known result of Erd\H{o}s. \begin{theorem}[Erd\H{o}s~\cite{Erdos}]\label{thm:erdos}
There exist $A,n_0>0$ such that for all $n>n_0$ there exists $n$-vertex graph $G$ which is triangle-free, does not contain a set of $[A\sqrt{n}\log{n}]$ independent vertices and the number of edges is at most $n^{3/2}/\sqrt{A}$.
\end{theorem}
In this theorem the graph $G$ is obtained as a  subgraph of $G(n,n^{3/2}/\sqrt{A})$, where recall that $G(n,m)$ denotes the random graph on $n$ vertices with $m$ edges.
Since $G$ has at most $n^{3/2}/\sqrt{A}< {n^{3/4} \choose 2}$ edges, it trivially cannot contain $K_{n^{3/4}}$ as a minor. Now by putting $t=n^{3/4}$ we can see that  Theorem~\ref{thm:main} is tight for $\varepsilon=1/3$ up to logarithmic factor, because $G$ described above has no $K_t$-minor, is triangle-free and the largest independent set  is of size less than $ [A\sqrt{n}\log{n}] = A\frac{n}{t^{1-1/3}}\log{n}.$

A line of research connecting Ramsey numbers to the minor-free classes of graphs was already pursued earlier; Walker~\cite{walker} in 1969, and, independently Steinberg and Tovey~\cite{Steinberg1993} in 1993 introduced the notion of planar Ramsey numbers. For $s,k$ integers, the \emph{planar  Ramsey number} $\textit{PR}(s, k )$ is the smallest integer $n$ for which every planar graph $G$ on $n$ vertices contains a copy of $K_s$ or its complement contains a copy of $K_k$. This is the usual Ramsey number except  the ground set is restricted to planar graphs.  Steinberg and Tovey~\cite{Steinberg1993}  determined all  planar Ramsey numbers and showed that they grow linearly  in contrary to usual Ramsey numbers which grow exponentially. This  exhibits that once restricted to the plane the Ramsey problem  becomes tractable. They determine all planar Ramsey numbers using the Four Colour Theorem, Gr\"unbaum's Theorem and by establishing a  positive solution to  a  conjecture of  Albertson, Bollob\'as, and Tucker~\cite{albertson1976};  every triangle-free planar graph on $n$ vertices contains an independent set of size $\lfloor n/3 \rfloor + 1$. The first two theorems are hard results in graph theory describing  the structural properties of a specific class of graphs.

Since the family of planar graphs consists precisely of those graphs which do not contain $K_5$ or $K_{3,3}$ as a minor, our question, that is, Conjecture~\ref{conj:stable} for triangle-free graphs, can be viewed as a natural generalization of  planar Ramsey numbers. For $t,s,k$ integers let us define the $t$-\emph{minor  Ramsey number} $\textit{MR}_t(s, k )$ to be the smallest integer $n$ for which every $K_t$-minor-free graph $G$ on $n$ vertices contains a copy of $K_s$ or its complement contains a copy of $K_k$.  In this language, Theorem~\ref{thm:main} says that for all  $0<\varepsilon<1/26$ there exist $t_0$ such that for all $t\geq t_0$,  $\textit{MR}_t(3,k)\leq  t^{1-\varepsilon}k$. Theorem~\ref{thm:erdos} says that  this bound cannot be improved further than $\varepsilon=1/3$ up to logarithmic factor. In Section~\ref{sec-open}, we conjecture that  for all $s\geq 4$,  a similar bound must hold: $\textit{MR}_t(s,k)\leq  t^{1-\varepsilon}k$ for sufficiently large $t$, here $\varepsilon$ cannot be larger than $1/s$. See Section~\ref{sec-open} for details.

In the next section, we prove various  lower bounds of the average degrees of  minors in dense graphs. These are slightly  technical results which we use  
 to prove Theorem~\ref{thm:main} in Section~\ref{sec-main}. We finish by discussing related open problems in Section~\ref{sec-open}.

\section{Average degrees of minors in dense graphs}

We need the following consequence of Chernoff's inequality.

\begin{lemma}\label{lemma-chernoff2}
Let $X_1$, \ldots, $X_m$ be independent random variables, each having value~$1$ with probability $p$
and value $0$ otherwise.  Then $\text{Pr}[\sum_{i=1}^m X_i>2pm]\le e^{-pm/3}$.
\end{lemma}

In this section, we derive a technical result (Lemma~\ref{lemma-incrdeg}) showing that minors with large average degree appear in graphs where many vertices
are joined by many internally disjoint paths of length at most three; this result and its proof are motivated by similar statements
used in~\cite{KuhnOsthusKss}.
For a graph $G$ and a positive integer $k$, let $G_{\leq3}^k$ denote the graph with the same vertex
set in which vertices $u$ and $v$ are adjacent if either $uv\in E(G)$ or $G$ contains $k$ internally vertex
disjoint paths of length at most $3$ from $u$ to $v$.  We will apply the following lemma twice in the proof of Theorem~\ref{thm:main},
once with $k$ almost linear in the average degree of $G$ (in Lemma~\ref{lemma-smallalpha}), once with $k=1$ (in Lemma~\ref{lemma-sd}).  Although most of the argument between these two cases
is shared, there turns out to be a qualitative difference reflected in the two possibilities for the choice of $b$ in the statement of the lemma.

\begin{lemma}\label{lemma-incrdeg}
Let $k$ be a positive integer, let $0<\varepsilon<1$ and $d$ be real numbers such that
$d\ge \max\{288^{1/(1-\varepsilon)}, 16\sqrt{k}\}$.  Let
$$b=\begin{cases}
2700d^{2+\varepsilon}/k&\text{ if $k\le 324d^{2\varepsilon}$}\\
150d^2/\sqrt{k}&\text{ if $k>324d^{2\varepsilon}$.}
\end{cases}$$
Let $G$ be a graph of minimum degree at least $d^{1-\varepsilon}/2$ and maximum degree at most $2d$.
If $G_{\leq3}^k$ has average degree at least $b$, then $G$ has a minor of average degree at least $d$.
\end{lemma}
\begin{proof}
Let $p=\max(18/d^{1-\varepsilon},\sqrt{k}/d)$; note that $p\le 1/16$ by the assumptions on $d$.  Let $X$ be a set of vertices of $G$ obtained by
picking each vertex independently at random with probability $p$.
We say that a vertex $z\in V(G)$ is \emph{blocked} if either $z\in X$ or $z$ has more than $2p(\deg(z)-2)+2$ neighbors in $X$.
Let $Z\subseteq V(G)\setminus\{z\}$ be a set of size two; let us bound the probability that $z$ is blocked under the condition that $Z\subseteq X$.
By Lemma~\ref{lemma-chernoff2}, this probability is at most $q\le p+e^{-p(\deg(z)-2)/3}\le p+e^{-pd^{1-\varepsilon}/6+2p/3}\le p+e^{-3+1/24}\le 1/8$.

We say that an edge $uv\in E(G_{\leq3}^k)$ is \emph{viable} if $u,v\in X$ and either $uv\in E(G)$, or $G$ contains
at least $k/2$ internally vertex disjoint paths of length at most $3$ from $u$ to $v$ whose internal vertices
are not blocked.  Let us give a lower bound on the probability that $uv$ is viable.  The probability that $\{u,v\}\subseteq X$ is $p^2$.
If $uv\not\in E(G)$, then by the definition of $G_{\leq3}^k$, there are at least $k$ internally vertex-disjoint paths of length at most three from $u$ to $v$ in $G$.
Under the condition that $\{u,v\}\subseteq X$, the probability that such a path contains an internal blocked vertex is at most $1/4$.
By Markov inequality, the probability (under the condition that $\{u,v\}\subseteq X$) that among $k$ internally vertex disjoint paths of length at most $3$,
more than $k/2$ contain an internal blocked vertex, is less than $\tfrac{k/4}{k/2}=\tfrac{1}{2}$.
Hence, the probability that $uv$ is viable is greater than $\tfrac{p^2}{2}$.
Let $H$ be an auxiliary graph with vertex set $X$ such that $uv$ is an edge of $H$ if and only if $uv\in E(G_{\leq3}^k)$ and $uv$ is viable.
We have $\E[\Vert H\Vert]>\tfrac{p^2}{2}\Vert G_{\leq3}^k\Vert\ge \tfrac{p^2b|G|}{4}$ and $\E[|H|]=p|G|$, and thus
$\E[\Vert H\Vert-pb|H|/4]>0$.

Consequently, there exists a choice of $X\subseteq V(G)$ such that $H$ has average degree greater than $pb/2$.
Let us fix such a choice of $X$.  Independently for each non-blocked vertex $z\in V(G)$ which has a neighbor in $X$,
choose one such neighbor $c_z\in X$ uniformly at random.  Since $z$ is not blocked, if $x\in X$ is a neighbor of $z$,
then the probability that $c_z=x$ is at least $\tfrac{1}{2p(\deg(z)-2)+2}\ge \tfrac{1}{4pd+2}$.

Let $G'$ be the minor of $G$ obtained by contracting all
edges $vc_v$ for non-blocked vertices $v$ adjacent to $X$ and by removing all other vertices not belonging to $X$.
Let us consider an edge $uv\in E(H)$, and let us estimate a probability that $uv\in E(G')$.
If $uv\in E(G)$, then we always have $uv\in E(G')$.  Suppose that $uv\in E(H)\setminus E(G)$, and
thus $G$ contains at least $k/2$ internally vertex disjoint paths of length at most $3$ from $u$ to $v$
whose internal vertices are not blocked.  Consider such a path $uxv$ or $uxyv$.  If $c_x=u$, and $c_y=v$ if $y$ exists,
then $G'$ contains an edge between $u$ and $v$.  The probability this happens is at least $\tfrac{1}{(4pd+2)^2}$.
Hence, the probability $p'$ that this happens for one of the at least $k/2$ paths between $u$ and $v$
is at least $1-\bigl(1-\tfrac{1}{(4pd+2)^2}\bigr)^{k/2}\ge 1-e^{-\frac{k}{2(4pd+2)^2}}$.
Since $p\ge \sqrt{k}/d$, we have $\tfrac{k}{2(4pd+2)^2}\le \tfrac{k}{32p^2d^2}\le \tfrac{1}{32}$, and since $e^{-x}\le 1-\tfrac{32}{33}x$ when $0\le x\le 1/32$,
we have $p'\ge \tfrac{16k}{33(4pd+2)^2}$.  Note that $pd\ge \sqrt{k}\ge 1$, and thus $4pd+2\le 6pd$ and $p'>\tfrac{k}{75p^2d^2}$. 

Since $H$ has average degree greater than $pb/2$, the minor $G'$ of $G$ has with positive probability
average degree greater than $p'pb/2=\tfrac{kb}{150pd^2}$.
If $k\le 324d^{2\varepsilon}$, then $p=18/d^{1-\varepsilon}$ and
the average degree is greater than $\tfrac{kb}{150pd^2}=\tfrac{kb}{2700d^{1+\varepsilon}}=d$.
If $k>324d^{2\varepsilon}$, then $p=\sqrt{k}/d$ and the average degree is greater than $\tfrac{kb}{150pd^2}=\tfrac{\sqrt{k}b}{150d}=d$.
\end{proof}

For a vertex $v$ of a graph $G$ and a positive integer $\ell$, let $N_G^\ell[v]$ denote the set of vertices of $G$ at distance at most
$\ell$ from $v$.  Let us note the special case of Lemma~\ref{lemma-incrdeg} when $k=1$.

\begin{lemma}\label{lemma-largen3}
Let $0<\varepsilon<1$ and $d\ge 288^{1/(1-\varepsilon)}$ be real numbers.
Let $G$ be a graph of minimum degree at least $d^{1-\varepsilon}/2$ and maximum degree at most $2d$.
If $|N_G^3[v]|\ge 2800d^{2+\varepsilon}$ for every vertex $v\in V(G)$, then $G$ has a minor of average degree at least $d$.
\end{lemma}
\begin{proof}
Let $k=1$ and note that $k<324d^{2\varepsilon}$.
By the assumptions, $G_{3,1}$ has minimum degree at least $2800d^{2+\varepsilon}-1\ge 2700d^{2+\varepsilon}/k$.
The claim thus follows by Lemma~\ref{lemma-incrdeg}.
\end{proof}

We also need the following reformulation of Tur\'an's theorem.
\begin{lemma}\label{lemma-turan}
A graph $G$ of average degree $d$ has an independent set of size at least $\tfrac{|G|}{d+1}$.
Equivalently, every graph has average degree at least $\tfrac{|G|}{\alpha(G)}-1$.
\end{lemma}

In a triangle-free graph $H$, it can be argued that if $Y$ is an independent set in $H_{\le3}^k$,
then the subgraph of $H$ induced by $N_H(Y)$ is sparse, and thus it contains a relatively large independent set.  If the independence number
of $H$ is bounded, this gives a bound on the independence number of $H_{\le3}^k$, which by Lemma~\ref{lemma-turan}
translates into a bound on the average degree of $H_{\le3}^k$.  This gives the following lower bound on the density of the graph $H_{\le3}^k$
for a triangle-free graph $H$ whose independence number is close to the minimum possible (about $\sqrt{|H|}$, as given by~\cite{ajtaikomlosszemeredi}).

\begin{lemma}\label{lemma-g3kdense}
Let $\beta,\gamma,d>0$ be real numbers such that $3\beta+2\gamma<1$ and $d^{2-2\beta-2\gamma}\ge 4$.
Let $k=\lfloor d^{1-3\beta-2\gamma}/16\rfloor$.
Let $H$ be a triangle-free graph of minimum degree at least $d^{1-\beta-\gamma}-1$.
If $|H|\ge d^{2-\gamma}$ and $\alpha(H)\le d^{1+\beta}$, then $H^k_{\le3}$ has average greater than $d^{2-2\beta-2\gamma}/4$.
\end{lemma}
\begin{proof}
Consider any non-empty independent set $Y$ in $H_{\le3}^k$.
Note that $Y$ is also an independent set in $H$.
Let $Z_1$ be the set of vertices of $H$ that have more than one neighbor in $Y$.
Since distinct vertices $x,y\in Y$ are non-adjacent in $H_{\le3}^k$, we have $|N_H(x)\cap N_H(y)|\le k-1$,
and
$$|N_H(x)\cap Z_1|\le \sum_{y\in Y\setminus \{x\}} |N_H(x)\cap N_H(y)|<(k-1)|Y|$$
for every $x\in Y$.  
Since $H$ has minimum degree at least $d^{1-\beta-\gamma}-1$, it follows that the number of vertices adjacent to $x$
and not to any other vertex of $Y$ satisfies
$$|N_H(x)\setminus Z_1|=|N_H(x)|-|N_H(x)\cap Z_1|>d^{1-\beta-\gamma}-1-(k-1)|Y|.$$
For distinct $x,y\in Y$, consider the subgraph $H_{xy}$
of $H$ induced by $(N_H(x)\cup N_H(y))\setminus Z_1$.  By the definition of $Z_1$, the sets $N_H(x)\setminus Z_1$
and $N_H(y)\setminus Z_1$ are disjoint, and since $H$ is triangle-free, they are independent; hence, the graph $H_{xy}$ is bipartite.
Since $H$ does not contain $k$ internally vertex-disjoint paths of length $3$ from $x$ to $y$, the subgraph $H_{xy}$
does not contain a matching of size $k$, and thus it has a vertex cover $Z_{xy}$ of size at most $k-1$.
Let $Z_2$ be the union of the sets $Z_{xy}$ over all distinct $x,y\in Y$; we have $|Z_2|<(k-1)|Y|^2$.
By the choice of $Z_1$ and $Z_2$, the set $N_H(Y)\setminus (Z_1\cup Z_2)$ is independent in $H$,
and since $\alpha(H)\le d^{1+\beta}$, we have
\begin{align*}
d^{1+\beta}&\ge |N_H(Y)\setminus (Z_1\cup Z_2)|=\Bigl(\sum_{x\in Y}|N_H(x)\setminus Z_1|\Bigr)-|Z_2|\\
&>|Y|(d^{1-\beta-\gamma}-1-(k-1)|Y|)-(k-1)|Y|^2> |Y|(d^{1-\beta-\gamma}-2k|Y|).
\end{align*}
By this inequality, we have $|Y|\neq \lceil 2d^{2\beta+\gamma}\rceil$, and thus $H_{\le3}^k$ does not
contain any independent set of size exactly $\lceil 2d^{2\beta+\gamma}\rceil$.  We conclude that $\alpha(H_{\le3}^k)<2d^{2\beta+\gamma}$.
By Lemma~\ref{lemma-turan}, it follows that $H_{\le3}^k$ has average degree
at least $\tfrac{|H_{\le 3}^k|}{\alpha(H_{\le 3}^k)}-1>\tfrac{d^{2-\gamma}}{2d^{2\beta+\gamma}}-1=d^{2-2\beta-2\gamma}/2-1\ge d^{2-2\beta-2\gamma}/4$.
\end{proof}

\section{Proof of Theorem~\ref{thm:main}}\label{sec-main}

Due to the following result of Thomason~\cite{thomasonexact}, as long as we do not care about polylogarithmic factors,
instead of considering clique minors, it suffices to consider minors with large average degree.
\begin{theorem}[Thomason~\cite{thomasonexact}]\label{thm-minor}
There exists a positive integer $t_0$ such that for every $t\ge t_0$, every graph of average
degree at least $\tfrac{1}{3}t\sqrt{\log t}$ contains $K_t$ as a minor.
\end{theorem}

As the first step, we combine Lemmas~\ref{lemma-incrdeg} and \ref{lemma-g3kdense} to show that if the independence number
of a triangle-free graph $G$ is close to the minimum possible (roughly $\sqrt{|G|}$), then $G$ contains a dense minor.

\begin{lemma}\label{lemma-smallalpha}
Let $\beta,\gamma,d>0$ be real numbers such that $7\beta+6\gamma<1$ and $d^{1-7\beta-6\gamma}\ge 10^7$.
Let $G$ be a triangle-free graph of maximum degree at most $2d$.
If $|G|\ge 2d^{2-\gamma}$ and $\alpha(G)\le d^{1+\beta}$, then $G$ has a minor of average degree at least $d$.
\end{lemma}
\begin{proof}
Let $d_0=\lfloor d^{1-\beta-\gamma} \rfloor\ge d^{1-\beta-\gamma}-1$.
Repeatedly remove from $G$ vertices of degree less than $d_0$ together with their neighborhoods, until we obtain
an induced subgraph $H$ of $G$ of minimum degree at least $d_0$.
Since $\alpha(G)\le d^{1+\beta}$, this process stops after at most $d^{1+\beta}$ vertices and their neighborhoods were removed, and thus
$|H|\ge |G|-d^{1+\beta}d_0\ge |G|-d^{2-\gamma}\ge d^{2-\gamma}$.

Let $k=\lfloor d^{1-3\beta-2\gamma}/16\rfloor$.  By Lemma~\ref{lemma-g3kdense},
the graph $H^k_{\le3}$ has average greater than $d^{2-2\beta-2\gamma}/4$.  Let $\varepsilon=\beta+\gamma$ and note that
$k>324d^{2\varepsilon}$.  Let $b=150d^2/\sqrt{k}$; since $k=\lfloor d^{1-3\beta-2\gamma}/16\rfloor\ge d^{1-3\beta-2\gamma}/25$,
we have $b\le 750d^{(3+3\beta+2\gamma)/2}$.  Since $d^{1-7\beta-6\gamma}\ge 10^7>3000^2$, this implies $b<d^{2-2\beta-2\gamma}/4$, and thus $H^k_{\le3}$
has average degree greater than $b$.
By Lemma~\ref{lemma-incrdeg} with $\varepsilon=\beta+\gamma$, $H$ (and thus also $G$) contains a minor of average degree at least $d$.
\end{proof}

When proving Theorem~\ref{thm:main} for a graph $G$, we can assume that every independent set $A$ in $G$ has relatively large
neighborhood; otherwise, the claim follows by applying induction to the graph $G-N_G[A]$ and combining the resulting independent
set with $A$.  With this in mind, let us conversely argue that in a triangle-free graph, if independent sets have large neighborhoods,
then the graph contains a dense minor.

\begin{lemma}\label{lemma-sd}
Let $0<\varepsilon<1/26$ and $d$ be real numbers such that $d^{1-26\varepsilon}\ge 10^7$.
Let $G$ be a triangle-free graph of maximum degree at most $2d$.  Suppose that for every vertex $v\in V(G)$, every non-empty independent
set $A\subseteq N_G^2[v]$ satisfies $|N_G[A]|>2800d^{1-\varepsilon}|A|$.
Then $G$ has a minor of average degree at least $d$.
\end{lemma}
\begin{proof}
Since $\{v\}\subseteq N_G^2[v]$ is an independent set for each $v\in V(G)$, we have $|N_G[v]|>2800d^{1-\varepsilon}$, and thus $G$ has minimum degree
greater than $d^{1-\varepsilon}/2$.  If $|N_G^3[v]|\ge 2800d^{2+\varepsilon}$ for every $v\in V(G)$,
then $G$ has a minor of average degree at least $d$ by Lemma~\ref{lemma-largen3}.  Therefore, we can assume that $G$ has a vertex $v$ such that
$|N_G^3[v]|\le 2800d^{2+\varepsilon}$.  Let us fix such a vertex $v$.

For every independent set $A\subseteq N_G^2[v]$, we have $N_G[A]\subseteq N_G^3[v]$; consequently,
$2800d^{1-\varepsilon}|A|<|N_G[A]|\le |N_G^3[v]|\le 2800d^{2+\varepsilon}$, and thus $|A|<d^{1+2\varepsilon}$.
Hence, the graph $H=G[N_G^2[v]]$ satisfies $\alpha(H)\le d^{1+2\varepsilon}$.  Since $G$ is triangle-free, $N_G(v)$ is an independent
set, and thus $|H|=|N_G[N_G(v)]|>2800d^{1-\varepsilon}|N_G(v)|>2d^{2-2\varepsilon}$.  By Lemma~\ref{lemma-smallalpha}
applied with $\beta=\gamma=2\varepsilon$, $H$ (and thus also $G$) has a minor of average degree at least $d$.
\end{proof}

Lemma~\ref{lemma-sd} enables us to inductively find large independent sets in graphs whose minors have average degree less than $d$.

\begin{lemma}\label{lemma-isbd}
Let $0<\varepsilon<1/26$ and $d$ be real numbers such that $d^{1-26\varepsilon}\ge 10^7$.
Let $G$ be a triangle-free graph of maximum degree at most $2d$.
If every minor of $G$ has average degree less than $d$, then
$\alpha(G)\ge \tfrac{|G|}{2800d^{1-\varepsilon}}$.
\end{lemma}
\begin{proof}
We prove the claim by induction on the number of vertices of $G$.  If $|G|\le 2800d^{1-\varepsilon}$, then the claim holds trivially,
since $\alpha(G)\ge 1$; hence, we can assume $|G|>2800d^{1-\varepsilon}$.
Since every minor of $G$ has average degree less than $d$, Lemma~\ref{lemma-sd} implies there exists a non-empty independent set $A$ in $G$
such that $|N_G[A]|\le 2800d^{1-\varepsilon}|A|$.
By the induction hypothesis, $G-N_G[A]$ has an independent set $A'$ of size at least $\tfrac{|G-N_G[A]|}{2800d^{1-\varepsilon}}\ge \tfrac{|G|}{2800d^{1-\varepsilon}}-|A|$,
and $A'\cup A$ is an independent set in $G$ of size at least $\tfrac{|G|}{2800d^{1-\varepsilon}}$.
\end{proof}

Let us now get rid of the assumption that the maximum degree of the considered graph is at most $2d$.

\begin{lemma}\label{lemma-finalboost}
Let $0<\varepsilon<1/26$ and $d$ be real numbers such that $d^{1-26\varepsilon}\ge 10^7$.
Let $G$ be a triangle-free graph.  If every minor of $G$ has average degree less than $d$, then
$\alpha(G)\ge \tfrac{|G|}{5600d^{1-\varepsilon}}$.
\end{lemma}
\begin{proof}
Let $Z$ be the set of vertices of $G$ of degree at least $2d$.
Since $G$ is a minor of itself, it by assumptions has average degree less than $d$, and thus $|Z|\le |G|/2$.
Since $G-Z$ has maximum degree less than $2d$, Lemma~\ref{lemma-isbd} implies that
$\alpha(G)\ge\alpha(G-Z)\ge \tfrac{|G-Z|}{2800d^{1-\varepsilon}}\ge \tfrac{|G|}{5600d^{1-\varepsilon}}$.
\end{proof}

Our main result now follows by a straightforward combination of Theorem~\ref{thm-minor} with Lemma~\ref{lemma-finalboost}.

\begin{proof}[Proof of Theorem~\ref{thm:main}]
Let $d=\tfrac{1}{3}t\sqrt{\log t}$ and let $\varepsilon'=(\varepsilon+1/26)/2$.  Note that $\varepsilon < \varepsilon' < 1/26$.
For sufficiently large $t$, we have $d^{1-26\varepsilon'}\ge 10^7$ and
$t^{1-\varepsilon}\ge 5600d^{1-\varepsilon'}$.  Since $G$ does not contain $K_t$ as a minor, Theorem~\ref{thm-minor}
implies that every minor of $G$ has average degree less than $d$, and thus
$\alpha(G)\ge\tfrac{|G|}{5600d^{1-\varepsilon'}}\ge \tfrac{|G|}{t^{1-\varepsilon}}$ by Lemma~\ref{lemma-finalboost}.
\end{proof}

\section{Concluding Remarks}\label{sec-open}

The \emph{fractional chromatic number} of a graph $G$, $\chi_f(G)$ is defined
to be the smallest positive real number $k$ for which there exists a probability distribution over
the independent sets of $G$ such that for each vertex $v$, given an independent
set $S$ drawn from the distribution, $\mathbb{P}[v\in S]\geq 1/k$. It is easy
to see that $\chi_f(G)\geq \tfrac{n}{\alpha(G)}$. Reed and
Seymour~\cite{reedseymour} strengthened the results of Duchet and Meyniel~\cite{duchetmeyniel}
by showing that graphs $G$ with no $K_{t+1}$ minor not only satisfy $\alpha(G)\ge \tfrac{|G|}{2t}$, but
actually they have fractional chromatic number at most $2t$.  It is natural to ask whether
Theorem~\ref{thm:main} can be extended in a similar way.

\begin{question}
Does there exist some $\varepsilon>0$ and a positive integer $t_0$ such that for
all $t\geq t_0$, all triangle-free graphs with no $K_t$ minor have fractional chromatic number at most
$t^{1-\varepsilon}$? 
\end{question}

Let us remark that our argument contains numerous steps that cannot be carried over to the fractional chromatic number setting (e.g.,
considering only the subgraph induced by vertices of degree at most $2d$).
Furthermore, note that the analogous question for regular chromatic number has a negative answer. By a result of the first author and
Kawarabayashi~\cite{dvorakkawarabayashi} there exist triangle-free graphs of treewidth at most $t$ (and thus
not containing $K_{t+2}$ as a minor) of chromatic number larger than $t/2$.

Another direction of research would be to consider Conjecture~\ref{conj:stable}
for graphs of clique number less than $s$, where $s\geq 4$ is a fixed integer. Based on the lower bounds on the Ramsey
numbers $R(s,k)$ for $s$ fixed and $k$ growing obtained by Bohman~\cite{bohman}
for $s=4$ and by Bohman and Keevash~\cite{bohmankeevash} for $s\geq 5$, one can
see that there are graphs with no $K_s$ subgraph and no independent set of order
$n^{2/s+1}$ up to a polylogarithmic factor. These graphs are also
$K_t$-minor-free for $t=n^{1-\frac{s-2}{s(s-1)-2}}$, which shows that for all such $s$ and
$t$ there are graphs which are $K_t$-minor-free, have clique number less than $s$,
and have no independent set of size $\tfrac{n}{t^{1-\frac{1}{s}}}$.  But we expect that similar result to
Theorem~\ref{thm:main} should hold here as well.

\begin{question}
Fix an integer $s\geq 4$. Is it true that for every $0<\varepsilon<1/s$ there exists positive integer $t_0$ such that for
all $t\geq t_0$, if $G$ is a $K_s$-free graph without $K_{t}$-minor, then $\alpha(G)\geq  \tfrac{|G|}{t^{1-\varepsilon}}$? 
\end{question}

A less ambitious goal would be to answer the above question for all $0<\varepsilon< \varepsilon_0$, for some  $\varepsilon_0>0$.

\section{Acknowledgments}
The second author would like to thank Katherine Edwards and Sergey Norin for
preliminary stimulating discussions on this problem during Bellairs Research
Workshop on Graph Theory, March 2017.

\bibliographystyle{acm}
\bibliography{indeptrfree}

\end{document}